\documentclass{ejmathsciPRE}

\usepackage{url}

\newtheorem{statement}[theorem]{Statement}

\begin{document}

\title{Some considerations in favor of the truth of \\Goldbach's Conjecture}

\author[C. D'Urso]{C. D'Urso\affil{1}\affil{,}\affil{2}}
\address{\affilnum{1}\ LUMSA University, 
Rome, Italy\\
\affilnum{2}\ Information Systems Development Office\\ 
ICT Department, Italian Senate,
Rome, Italy}

\email{{\tt cirodurso@ieee.org}\ (C. D'Urso)}

\begin{abstract}
This article presents some considerations about the Goldbach's conjecture (GC). The work is based on analytic results of the number theory and it provides a constructive method that permits, given an even integer, to find at least a pair of prime numbers according to the GC. It will be shown how the method can be implemented by an algorithm coded in a high-level language for numerical computation.  Eventually a correlation will be provided between this constructive method and a class of problems of operations research.
\end{abstract}


\maketitle

\section{Introduction}
In his famous letter to Leonhard Euler dated 7 June 1742 \cite{key-2}, Christian Goldbach stated his famous conjecture that “every integer greater than 2 can be written as the sum of three primes”, statement that is equivalent to that every even number is a sum of two primes, as Euler stated in the letter dated 30 June 1742 ('I regard this as a completely certain theorem, although I cannot prove it'). This strong Goldbach conjecture implies the conjecture that all odd numbers greater than 7 are the sum of three odd primes, which is known today variously as the "weak" Goldbach conjecture, the "odd" Goldbach conjecture, or the "ternary" Goldbach conjecture. 

In 1923 the ternary conjecture has been proved under the assumption of the truth of the generalized Riemann hypothesis \cite{key-8}, in 1937 Vinogradov\cite{key-4} removed the dependence on the Riemann Hypothesis, and proved that this it true for all sufficiently large odd integers. In 1956 Borodzkin showed that an integer greater than a large integer is sufficient in Vinogradov's proof. In 1989 and 1996 Chen and Wang reduced this bound \cite{key-10}. Using Vinogradov's method, Chudakov,van der Corput, and Estermann \cite{key-5,key-6,key-7} showed that almost all even numbers can be written as the sum of two primes (in the sense that the fraction of even numbers which can be so written tends towards 1). In a very recent paper mathematician Terence Tao of the University of California, Los Angeles, has inched toward a proof, in fact he has shown that one can write odd numbers as sums of, at most, five primes \cite{key-113}.
Although believed to be true, the binary Goldbach conjecture is still lacking a proof. We state it as follows.

\begin{statement}\label{s1}
 Every even integer greater than 4 can be written as the sum of two primes.
\end{statement}

\section{Alternative statements of the problem} \label{alt}

Consider an even natural number $2e$, we can state the conjecture in term of the natural number $e$
  as follows.

\begin{statement}\label{s2}
For every integer $e>3$  there exists a couple of odd prime numbers $q{}_{1}$  and $q{}_{2}$  such that $e$  is their semi-sum, that is: $e = \frac{q{}_{1}+q_{2}}{2}.$
 \end{statement}

This statement is equivalent to the original one considering that $2e$, with e  a generic integer, is always an even natural, and that the original statement is trivially verified for the numbers 4 and 6.

\begin{definition}\label{d1}
Two prime numbers $q{}_{1}$  and $q{}_{2}$, with $q_{2}>e>q_{1}$, are called mirror primes respect to an integer $e$  if there exists an integer $d$  such that: $e-q_{1}=q_{2}-e=d$.
 \end{definition}

So another equivalent statement to 2.1, and of course to 1.1, is:
\begin{statement}\label{s3}
 For every natural $e>3$  there exists a pair of mirror primes respect to $e$. 
 \end{statement}
In the following we will consider, without loss of generality, $e>7$.
The statement \ref{s3}, leads us to analyze when, given a natural number $e$, there exists at least one integer $d$  defining two mirror primes respect to $e$. First of all we have to ask if and when there exist such number $d$  (if $e$ is prime the choice is trivially $d=0$). Moreover we are interested in an integer such that the two integers $e\pm d$  are primes (observe that the primality can be checked by means of an efficient algorithm currently available \cite{key-1}). 

\section{Notation} \label{nt}
Let observe that in the following we use that notation:

With the symbol $\pi(x)$  we denote the number of primes less than or equal to $x$.

With the expression $a\equiv b\:(mod\, c)$  we state that $a$  is congruento to $b$  modulo $c$.

With the expression $a\equiv b\:(mod\, c)$  we state that $a$  is not congruento to $b$  modulo $c$.

With the expression $(n,m]$  we indicate the interval open on left and close on right.

With the expression $a\nmid b$  we denote that $a$  doesn't divide $b$.

With the exprression $|\{a_{i}\}|$ we denote the cardinality of the set of elements $\{a_{i}\}$.

In some passages the symbol $[x]$  indicates the greatest integer less then or equal to $x$, but in general the square brackets are used to group symbols in an expression. 

\section{Plan of the work} \label{pp}
The aim of the article is to calculate the number $d$  such that the two integers $e\pm d$ were primes. We observe that if the number $d$  is such that it satisfies the $\pi(\sqrt{2e})$  relations $d\, not\equiv\pm e\left(mod\, p_{i}\right)$ this fact alone represents a sufficient condition for the desired result. Therefore the results provided in the following sections can be summarized by these two logical steps:

1. Given an even natural number $2e$  it will be provided a method in order to calculate an integer $d$  such that $e\pm d\, not\equiv0\,(mod\, p_{i})$  where $p_{i}$ are the $\pi(\sqrt{2e})$ primes less than or equal to $\sqrt{2e}$. This assure that the two integers $e\pm d$  are prime numbers. This step depends on the choice of a set of positive integers $b{}_{i}$
and on the choice of their signs (+/-). 

2. It will be provided a method in order to calculate at least one set of integer $b{}_{i}$  such that $b{}_{i}not \equiv\pm e\:(mod\, p_{i})$, and considering that $d\equiv b_{i}\:(mod\, p_{i})$ the consequence is that $d\, not\equiv\pm e\left(mod\, p_{i}\right)$. This step, in combination with the step 1, allow us to obtain the number $d$  greater than $-e$
 and less than $e$. Consequently we obtain a pair of mirror primes respect to $e$, that is: $(e-d) + (e+d) = 2e$ .

\section{The Chinese Remainder Theorem and the number $d$} \label{crt}
About the existence of prime numbers in the intervals definened by the number $e$, that is: $\left(0,e\right)$ and $\left(e,2e\right)$, the following result holds. 
\begin{lemma} \label{l1}
Given an arbitrary integer $e>7$ there is at least one prime number in the intervals $\left(0,e\right)$ and $\left(e,2e-2\right)$.
\end{lemma}
\begin{proof} In the first interval we have trivially at least the first two primes. Regarding the second interval, actually it is the Bertrand's postulate, also called the Bertrand-Chebyshev theorem or Chebyshev's theorem: it states that if $e>3$, there is always at least one prime $p$  between $e$  and $2e-2$ \cite{key-103}.$\square$
\end{proof}

Furthemore we can ask if there are particular relations among $q{}_{1}, q_{2}, e, and d$. We can easily observe that the mirror primes can not be prime factors neither of the natural number $2e$ nor of $d$. We summarize those remarks in the following Lemma. 
 \begin{lemma} \label{l2}
Given $e$, even integer number, for all pairs of mirror primes respect to $e$ it holds: 
$\mathit{(a)}  q{}_{1} \nmid  e; 
\mathit{(b)}  q{}_{2} \nmid 2e; 
\mathit{(c)}  q{}_{1} \nmid d.$
\end{lemma}
\begin{proof} We prove $\mathit{(a)}$ by reductio ad absurdum. We have: $q{}_{2}= 2e - q{}_{1}$, and if $q{}_{1} \mid e$ then $q{}_{1} \mid q{}_{2}$, but $q{}_{2}$ is prime and so it is impossible. The statement $\mathit{(b)}$ follows from the identity $q{}_{1} = 2e - q{}_{2}$ and from the fact that $q{}_{1}$ is by hypothesis prime. Similarly the third statement follows from observing that $q{}_{2} = q{}_{1}+ 2d.\square$
\end{proof}

Remembering that a prime number is a natural number greater than 1 that has no positive divisors other than 1 and itself, it is easily to infer that every composit number $n$ has a prime factor less than or equal to $\sqrt{n}$. In fact if $n$ has a prime factor less than or equal to its square root it is a composite number. Conversely let suppose for simplicity that $n$ could be written, according to the fundamental theorem of arithmetic \cite{key-2,key-102}, as the product of two primes $a$ and $b$. If both of these two numbers were greater than $\sqrt{n}$ we should have that: $n=a\cdot b>\sqrt{n}\cdot\sqrt{n}=n$, that is $n>n$, and it is not possible. The underlying rationale can be generalized to an arbitrary number of prime factors.
So at least one prime divisors of each of the two quantity $e\pm d$, if composite, are in the interval $\left[1,\sqrt{e\pm d}\right]$. On the other hand if $e\pm d$ are not divisible by any prime less then or equal to $\sqrt{e\pm d}<\sqrt{2e}$ then they can not be a composite integer, so it has to be prime.
The following lemma introduces the calculation method for the number d depending on an arbitrary choice of suitable integers $b_{i}$.

\begin{lemma} \label{l3}
Given a natural number $e$ and a set of $k=\pi(\sqrt{2e})$ arbitrary integers $b{}_{i}$, there exists an integer number $d$ solution of the following system of congruencies:

\begin{center}
$d\equiv b_{i}\:(mod\, p_{i}),\: i=1..k$ 
\end{center}

where $p{}_{i}$ are the $k=\pi(\sqrt{2e})$ primes less then or equal to $\sqrt{2e}$.
\end{lemma}

\begin{proof}
 Observing that $\left(p{}_{i},p_{j}\right)=1$ for $i\neq j$, the Chinese Remainder Theorem (CRT) \cite{key-101} tell us that the system has a solution congruent to $p{}_{1}\cdot p_{2}\cdot...\cdot p_{k}$. The solution can be calculated by the following formula:
$d =\underset{1\leq i\leq k}{\sum}b_{i}P_{i}P'_{i}$, where: $P=p{}_{1}\cdot p_{2}\cdot...\cdot p_{k} ,P{}_{i}=P/p_{i}$, and $P'_{i}$ is the inverse of $P_{i}$, that is: $P{}_{i}P'{}_{i}\equiv1(mod\: p_{i}).\square$ \hspace*{\fill}
\end{proof}

 Let note that the right terms of several congruencies could be considered as negative integers, in this case the result of the lemma is still valid (that is, supposing $b_{i}>0$, we would have: $d\equiv-1\cdot b_{i}\:(mod\, p_{i})$ for some $i\in[1..k]$, and the correspondent terms of the expession of the solution would be modified accordingly multiplying by $-1$).
Supposing that we can choose a set of integer numbers $b{}_{i}$ with $i=1..\pi\left(\sqrt{2e}\right)$ such that $b{}_{i}not\equiv\pm e\left(mod\, p_{i}\right)$, then the number $d$ calculated by the foregoing method, if it is in the interval $\left[-(e-2),+(e-2)\right]$, is such that:
$d\equiv b_{i}\:(mod\, p_{i}),\, and\: b{}_{i}not\equiv\pm e\left(mod\, p_{i}\right),\:\Rightarrow d\, not\equiv\pm e\left(mod\, p_{i}\right)$, that guarantees $e\pm d$ being a pair of primes. \hspace*{\fill}

\section{How to calculate the integers $b{}_{i}$} \label{bi}
The system of congruencies defined in the  Lemma  \ref{l3}  provides an integer $d$ that depends on the choice of the $k$ integers $b{}_{i}$. Let us try to state the requirements they must meet. First of all we want the followings $k$ conditions hold: $b{}_{i}not \equiv\pm e\:(mod\, p_{i})$, in order to assure that $e\pm d\, not\equiv0\,(mod\,p_{i})$, that is $p_{i}\nmid e\pm d$, which ensures that $e\pm d$ will be prime numbers. A possible choice for the number $b{}_{i}$ could be that provided by the following lemma.

\begin{lemma} \label{l4}
Let $e$ be a positive integer, we denote $p{}_{1},p_{2},...p_{k}$ the first $k$ primes with $k=\pi(\sqrt{2e})$, then the quantity defined as follows are not congruent to $\pm$ e modulo $p_{i}$ for each $i : b{}_{i}=\left[\frac{e}{p_{i}}\right]p_{i}$, if $ p_{i}\nmid e$, and $b{}_{i}=\left[\frac{e}{p_{i}^{_{\alpha_{i}}}}\right] p_{i}+1$, if $p{}_{i}\mid e$ with $\alpha{}_{i}$ the i-th prime power according to the fundamental theorem of arithmetic \cite{key-2, key-102}, and $\left[x\right]$ indicates the greatest integer less then or equal to $x$.
\end{lemma}

\begin{proof} First consider the case in which $p_{i}\nmid e$. Observe that the expression $e-\frac{b_{i}}{p_{i}}$ represents the remainder after the division $\frac{e}{p_{i}}$, so it is not divisible by $p_{i}$. If $e\pm b_{i}$ was divisible by $p_{i}$ we could write: $e=(h\mp\left[\frac{e}{p_{i}}\right])p_{i}$, with $h$ a generic integer, but this would mean that $p{}_{i}\mid e$, contra hypothesis. Now suppose $p_{i}\mid e$, similarly if $e\pm b_{i}$ was divisible by $p_{i}$ we could write: $e=(h\mp\left[\frac{e}{p_{i}^{_{\alpha_{i}}}}\right])p_{i}+1$, with $h$ a generic integer, but but this would mean that $p{}_{i}\nmid e$, contra hypothesis.$\square$
\end{proof}

Unfortunately the choice of the numbers $b_{i}$ as indicated by lemma \ref{l4} doesn't assure that $d$ is less than our integer $e$, let see for instance the following sidebar related to the integer 68.
\\
\framebox{\begin{minipage}[t]{1\columnwidth}%
\begin{center}
$e=68$
\par\end{center}

\begin{center}
\begin{tabular}{|c|c|c|c|}
\hline 
$p{}_{i}$ & $P{}_{i}$ & $P{}_{i}'$ & $b{}_{i}$\tabularnewline
\hline 
\hline 
2 & 1155 & 1 & 35\tabularnewline
\hline 
3 & 770 & 2 & 66\tabularnewline
\hline 
5 & 462 & 3 & 65\tabularnewline
\hline 
7 & 330 & 1 & 63\tabularnewline
\hline 
11 & 210 & 1 & 66\tabularnewline
\hline 
\end{tabular} 
\par\end{center}

\begin{center}
$d=266805\,(mod\,2310)=1155$. $p{}_{1}=d-e=1087$, $p{}_{2}=d+e=1223$,
both prime numbers.
\par\end{center}%
\end{minipage}}
\begin{center}{\label{e68}Sidebar 1. Lemma 4 applied to the number $e=68$}
 \end{center}  

In order to avoid values of $d$ greater than or equal to $e-2$ and, because the solution $d$ could be negative depending on the choice of the sign of $b_{i}$, less than or equal to $-(e-2)$, the following constraints must hold:
\\
\begin{center}
$Q_{k}\cdot B^{T}\leq e-2$ 
\\
$Q_{k}\cdot B^{T}\geq-(e-2)$

where:

$B_{1\times k}= \left(\pm b{}_{1}\:\pm b_{2}\:\pm b_{3}...\pm b_{k}\right), \;$  
\\

$Q_{k\times k}=\begin{bmatrix}P{}_{1}\cdot P_{1}' & 0 & 0 & ... & 0\\P{}_{1}\cdot P_{1}' & P{}_{2}\cdot P_{2}' & 0 & ... & 0\\... & ... & ... & ... & ...\\P{}_{1}\cdot P_{1}' & P{}_{2}\cdot P_{2}' & P{}_{3}\cdot P_{3}' & ... & P{}_{k}\cdot P_{k}'\end{bmatrix}$,  \\
with $det(Q)\neq0.$  
\end{center}
Observe that the solution $d$ of the system of congruencies of lemma \ref{l3} is $Q_{k}\cdot B^{T}$, where $Q_{k}$ is the k-th row of the matrix $Q$, where $k=\pi(\sqrt{2e})$.
\\
In order to  calculate the $k$ sets of integer $b{}_{i}$ such that $b{}_{i}not \equiv\pm e\:(mod\, p_{i})$ we define the following sets of indexes $i$:
\begin{definition} \label{d2} Let define the sets: $I{}_{S}=\{i: p{}_{i}\mid e\}$ and $I{}_{R}=\{i: p{}_{i}\nmid e\}$, where $1\le i \le \pi(\sqrt{2e})$. Regarding the choice of $b{}_{i}$ we have: \\
If the index $i$ is in $I{}_{S}$ then we choose: $b{}_{i} \in  \{ 1 \le j \le e-2:  p{}_{i} \nmid ( e\pm j ) \}$, \\
If the index $i$ is in $I{}_{R}$ then we chose: $b{}_{i} \in \{0\} \cup \{ 1 \le j \le e-2:  p{}_{i} \nmid ( e\pm j ) \}.$ 

From these sets  $\{b_{j}\}_{i}, i=1..k$, considering that for each positive element $+b_{ij}$ there is also the corrisponding negative one $b_{ij+1} = $ $-b_{ij}$, we can define the correspondet sets $\{w_{j}\}_{i}$ as $w_{i j}=b_{ij}\cdot P_{i}\cdot P'_{i},\: i=1..k, j= 1..|\{b_{j}\}_{i}|$.
\end{definition}

\begin{remark}\label{r1}
Given an arbitrary integer $e>7$, if there exists a choice of the quantities $b_{i}$, calculated according to definition \ref{d2} (so that, if not equal to zero, they are positive and negative quantities), that satisfies the following relations:

\begin{center}
\begin{equation}
Q_{k}\cdot B^{T}\leq e-2
 \end{equation}
\begin{equation}
Q_{k}\cdot B^{T}\geq-(e-2)
\end{equation}
\end{center}

where $d=Q_{k}\cdot B^{T}$ is the solution of the system of $k$ congruencies: $d\equiv b_{i}\:(mod\, p_{i})$, then the two integers $e-d$ and $e+d$ are prime numbers, that is there exists a pair of mirror primes respect to the integer $\mathit{e}$.
\end{remark}

There is an obvious way to formulate the problem as a constraint satisfaction problem (CSP)  \cite{key-150}  \cite{key-152}  \cite{key-154}, that is the research of a value, selected from a given finite domain, to be assigned to each variable so that all constraints relating the variables are satisfied. Defining $x_{ij}=1$ if it is chosen the value $w_{ij}$ we have:

\begin{center}
\begin{equation}
d=\underset{1\leq i\leq k}{\sum} \underset{1\leq j\leq |\{b_{h}\}_{i}|}{\sum}w_{ij}\cdot x_{ij}
\end{equation}
\begin{equation}
d \in [-(e-2), (e-2)]
\end{equation}
\begin{equation}
\underset{1\leq j\leq |\{b_{h}\}_{i}|}{\sum}x_{ij}=1, \forall i
\end{equation}
\begin{equation}
x_{ij} \in \{0, 1\}
\end{equation}
\end{center}

The definition \ref{d2} implies that if $e$ is a prime number then all the sets ${\{b{}_{j}\}}_i$ have an element equal to zero, and the choice of $d$ is trivially $d=0$. The algorithm provided in the following check this zero-configuration as the first choice so it is able to determine the primality of the given number.
The enormous number of combinations of the elements $b{}_{i}$ that is possible to choose could be explored by a simple backtracking algorithm, and an example is priovided by the pseudo-code in the  sidebar 2.\\

Obviously this strategy is not an efficient option when the number $e$ becomes large, in fact the search space becomes too large to search exhaustively. One way to reduce the combinations explored consists in gradually expanding the research domain, for example by one unit at each step for all the $k$ sets, depending on a random choice. We can also ignore unfeasible nodes (combinations for which the number $d$ is not in the feasible interval) observing that if $d$ is greater than $e-2$ when the current $w$ is greater than zero we can jump to the next value of $w$ (negative by definition). The algorithm modified in this way has been proven to be already more efficient for small number as reported in table 1. That is probably due to the fact that the first combination of number $w_{ih}$ that satisfied both of the inequalities 1 and 2 is obtained for indexes $j <<  |\{b_{j}\}_{i}|$. This is just a conjecture, directly derived from the original GC, namely a stronger form of GC. Once verified by computer that the new conjecture is indeed satisfied up to some enormous number, it may be easier to analyze this stronger form than the original GC. See Table 1 for a comparison of the two algorithms in terms of execution time for some small numbers (algorithms executed on an OiS platform with Intel Core 2 Duo SL9600).
So this conjecture can be writte as:
\begin{statement}\label{s4}
For every integer $e>7$  if there exists one or more choices of ${w_{ij}}$ such that $d = \underset{1\leq i\leq k}{\sum w_{i}}$ is in the interval $[-(e-2), (e-2)]$, then for at least one choice we have that all the k numbers $|b_{i}|$ of the terms ${w_{i}}$  ( $|x|$ means here the absolute value of $x$) are less than or equal to $\sqrt[k/r]{e}$, with $r$ a real less than $k$.
 \end{statement}

This upper bound for the number of element of each set $\{b_{j}\}_{i}$ imply that the run-time complexity of the second algorithm is of $O(e^r)$, with $r < k$, much less than the case of the first one.

\begin{remark}\label{r2}
It could be tried to prove the conjecture by induction. If we consider the base case as $e=8$, we have that $d=3$ is obtained by the first and the second $b_{i}$, in particular as $b_{11}=1$, $b_{22}=-3$, $d=3$ and $r=2/3$. If we suppose the conjecture true for $e$ we have that for $e+1$ holds for $ i\in [1,k(e)]$:
\begin{center} 
$P_{i}(e+1)=P_{i}(e)$, if $k(e)=k(e+1)$\\
$P_{i}(e+1)=P_{i}(e) p_{k(e+1)}$, if $k(e)<k(e+1)$
\end{center} 
In the first case it is easy to demonstrate that for each indexes in the definition 2 related to the integer $e+1$, it can be chosen either the same index as for $e$ or one of the indexes $j-1$ or $1-j$, so that the number $d$ is still in the desiderable interval, and  the fact that the k numbers $|b_{i}|$ for the number $e$ are less than or equal to $\sqrt[k/r]{e}$ implies that the k numbers $|b_{i}|$ for the number $e+1$ are less than or equal to $\sqrt[k/r]{e+1}$.
The second case is more tricky, and we decided to attack it in a next research.
\end{remark}

\framebox{\begin{minipage}[t]{1\columnwidth}
\vspace{0pt}
 Inizialize $d$, the vector $c$ (pointer to the current elemens of the set $\{w_{i}\}_{h}, h=1..k$)\\
{\bf while}   $d \notin [-(e-2), (e-2)]$
\begin{quote}
h=k; flag=1;\\
	{\bf while} (flag==1)
\begin{quote}
		$d = d - {w_{h,c(h)}}$;\\
		c(h) = c(h) + 1;  /move forward \\
		{\bf if} c(h) $> $ $ |\{b_{i}\}_{h}|$
\begin{quote}
                                flag = 1;\\
                                c(h) = 1;\\
			$d = d + {w_{h,c(h)}}$;\\
			h = h - 1;
\end{quote}
		{\bf else}
\begin{quote}
			flag = 0;\\
			$d = d + {w_{h,c(h)}}$;
\end{quote}
		{\bf end}
\end{quote}
	{\bf endwhile}
\end{quote}
{\bf endwhile}

\end{minipage}}

\begin{center}{\label{deepbf}Sidebar 2. Algorithm based on a backtracking  'depth-first' search}
 \end{center}  

The core of the algorithms provided so far can be applied together with a simple heuristic as well. In particular we can focus the search in a neighborhood of the values $w_{i}$ choosen such that they be of the same order of $w_{1}$, for all the $w_{1j}$.
Also the ordering of the set of $w$ can influence the performance of the search, and we can order the rows $w_{i}$ in descending order respecting to the mean of the elements of each row so that the first row contains the biggest values on average.  $i$.

\framebox{\begin{minipage}[t]{1\columnwidth}
\vspace{0pt}\raggedright
Inizialize $d$, the vector $c$  (pointer to the current elemens of the set $\{w_{i}\}_{h}, h=1..k$), the vector $depth$ \\
{\bf while} $d \notin [-(e-2), (e-2)]$
\begin{quote}
$depth_{i}=min(depth_{i} + 2 * rand(1,k),  |\{b_{j}\}_{i}|)$;\\
{\bf while}  (h$>$0)
\begin{quote}
	h = k;
	flag = 1;\\
	{\bf while} (flag==1)
\begin{quote}
		$d = d - {w_{h,c(h)}}$;\\
		c(h) = c(h) + 1;\\
		{\bf while} the node has been already explored
		\begin{quote}
                                    c(h)=c(h)+1;  /move forward
		\end{quote}
		{\bf endwhile}\\
		{\bf if} c(h) $>$ depth(h)
\begin{quote}
                                flag = 1;\\
                                c(h) = 1;\\
			$d = d + {w_{h,c(h)}}$;\\
			h = h - 1;
\end{quote}
		{\bf else}
\begin{quote}
			flag = 0;\\
			{\bf if} a unfeasible node has been reached {\bf then} 
			\begin{quote}
			c(h)=c(h)+1; /move forward
			\end{quote}	
			$d = d + {w_{h,c(h)}}$;
\end{quote}
		{\bf end}
\end{quote}
	{\bf endwhile}
\end{quote}
{\bf endwhile}
\end{quote}
{\bf endwhile}

\end{minipage}}

\begin{center}{\label{deepbf}Sidebar 3. Algorithm based on a forward checking approach}
 \end{center}

\framebox{\begin{minipage}[t]{1\columnwidth}

Inizialize $d$, the vector $c$  (pointer to the current elemens of the set $\{w_{i}\}_{h}, h=1..k$), the vector $depth$ \\
{\bf while} $d \notin [-(e-2), (e-2)]$

\begin{quote}
c(1)=c(1)+1;\\
/* Let's consider just the positive value of ${w_{h,c(h)}}$:\\
{\bf while}  $ ({wabs_{h,c(h)}} < 0.98 * {wabs_{1,c(1)}})$ AND $ ( c(h) <   |\{b_{j}\}_{h}|)$
\begin{quote}
c(h) = c(h) + 1;
\end{quote}
{\bf endwhile} 
\\
depth(h)=c(h)-1;\\
{\bf while}  (h$>$0)
\begin{quote}
	h = k;
	flag = 1;\\
	{\bf while} (flag==1)
\begin{quote}
		$d = d - {w_{h,c(h)}}$;\\
		c(h) = c(h) + 1;\\
		{\bf if} c(h) $>$ depth(h)+2
\begin{quote}
                                flag = 1;\\
                                c(h) = 1;\\
			$d = d + {w_{h,c(h)}}$;\\
			h = h - 1;
\end{quote}
		{\bf else}
\begin{quote}
			flag = 0;\\
			$d = d + {w_{h,c(h)}}$;
\end{quote}
		{\bf end}
\end{quote}
	{\bf endwhile}
\end{quote}
{\bf endwhile}
\end{quote}
{\bf endwhile}

\end{minipage}}

\begin{center}
{\label{deepbf}Sidebar 4. Algorithm based on a simple heuristic that doesn't explore the entire tree}
 \end{center}

\begin{table}[htbp]

\centering
\fbox{%
\begin{tabular}{c|cc|cc}
& \multicolumn{2}{c|}{$Algorithm_{1}$} & \multicolumn{2}{c}{$Algorithm_{2}$} \\
Number e & $t (sec)$ & $ d$ & $t (sec)$ & $d$\\ \hline
$68$ & $<1$ & $15$ & $<1$ & $15$ \\
$188$ & $249$ & $-105$ & $132$ & $-105$ \\
$273$ & $>3600$ & $-$ & $646$ & $206$ \\
$368$ & $>3600$ & $-$ & $590$ & $-231$
\end{tabular}}
\caption{\label{tab_simmses} Comparison of the two algorithms for some small integers}
\end{table}

\newpage
Another way to view the problem is the following.
In order to calculate the quantity $d$ according to the lemma  \ref{l3}, we have to calculate the set of $b_{i}$ (positives by construction) as stated in definition \ref{d2}. Once obtained these positive integers, we have to attribute the appropriate sign to each one. The choice of the integer $+b_{i}$ versus $-b_{i}$ can be viewed as the choice of two sets of binary variable $x_{i}$ and $y_{i}$, where:
\begin{center}
$x_{i}=\begin{cases}
0 & \Leftrightarrow y_{i}=1\Leftrightarrow we\: choose\:-b_{i}\\
1 & \Leftrightarrow y_{i}=0\Leftrightarrow we\: choose\:+b_{i}
\end{cases}$
 \end{center}  
The constrains (1) and (2) can be written as:
\begin{center}
\begin{equation}
\underset{1\leq i\leq k}{\sum}w_{i}\cdot x_{i}+\underset{1\leq i\leq k}{\sum(-w_{i})\cdot y_{i}}\leq e-2
\end{equation}
\begin{equation}
\underset{1\leq i\leq k}{\sum}w_{i}\cdot x_{i}+\underset{1\leq i\leq k}{\sum(-w_{i})\cdot y_{i}}\geq-(e-2)
\end{equation}
 \end{center}  
where:
\begin{center}
$w_{i}=b_{i}\cdot P_{i}\cdot P'_{i},\: i=1..k$.
\end{center}  

It may be observed that $y_{i}=1-x_{i}$, and then the preceding constrains, after substituting $y_{i}=1-x_{i}$, can be written as (in the following the symbol [.] represents just parentheses):
\begin{center}
\begin{equation}
\underset{1\leq i\leq k}{\sum}w_{i}\cdot x_{i}\leq\frac{1}{2}\left[(e-2)+\underset{1\leq i\leq k}{\sum w_{i}}\right]
\end{equation}
\begin{equation}
\underset{1\leq i\leq k}{\sum}w_{i}\cdot x_{i}\geq\frac{1}{2}\left[-(e-2)+\underset{1\leq i\leq k}{\sum w_{i}}\right]
\end{equation}
\end{center} 

Therefore it must be proved that, under the assumption that $e\geq3$, the feasible set defined by the constraints is not empty. In other words, at least one choice of the binary variables $x_{i}$ satisfies:
\begin{center}
\begin{equation}
\underset{1\leq i\leq k}{\sum}w_{i}\cdot x_{i}\leq\frac{1}{2}\left[(e-2)+\underset{1\leq i\leq k}{\sum w_{i}}\right]=\frac{1}{2}\underset{1\leq i\leq k}{\sum}w_{i}+E
\end{equation}
\begin{equation}
\underset{1\leq i\leq k}{\sum}w_{i}\cdot x_{i}\geq\frac{1}{2}\left[-(e-2)+\underset{1\leq i\leq k}{\sum w_{i}}\right]=\frac{1}{2}\underset{1\leq i\leq k}{\sum}w_{i}-E
\end{equation}
\end{center} 
where $E=\frac{(e-2)}{2}$ .

Now observe that the two constraints define a convex set. In fact, for all $\alpha\in(0,1)$ and considering two elements of this set $x_{i}^{1}$ and $x_{i}^{2}$ satisfying the constraints, we have:
\begin{center}
$\underset{1\leq i\leq k}{\sum}w_{i}\cdot(\alpha x_{i}^{1}+(1-\alpha)x_{i}^{2})=\alpha\underset{1\leq i\leq k}{\sum}w_{i}\cdot x_{i}^{1}+(1-\alpha)\underset{1\leq i\leq k}{\sum}w_{i}\cdot x_{i}^{2}\leq\frac{1}{2}\underset{1\leq i\leq k}{\sum}w_{i}+E$\\
 
$\underset{1\leq i\leq k}{\sum}w_{i}\cdot(\alpha x_{i}^{1}+(1-\alpha)x_{i}^{2})=\alpha\underset{1\leq i\leq k}{\sum}w_{i}\cdot x_{i}^{1}+(1-\alpha)\underset{1\leq i\leq k}{\sum}w_{i}\cdot x_{i}^{2}\geq\frac{1}{2}\underset{1\leq i\leq k}{\sum}w_{i}-E$
\end{center} 

The problem to find a solution in this set is known in literature as “Convex Feasibility Problem” \cite{key-109,key-110,key-111}. It consists in finding a point in the intersection of convex sets. The common way for solving it is the relaxation-projection algorithm \cite{key-111,key-112}.

Howevever we are interested for our goal in proving the existence of a choice of the variables $x_{i}$. Let observe that we may order the terms $w_{i}$ such that: $w_{i}\geq w_{i+1},\forall i=1..k-1$, and then we can determine the index $h$ such that:
\begin{center}

\begin{equation}
\underset{1\leq i\leq h-1}{\sum}w_{i}<\frac{1}{2}\underset{1\leq i\leq k}{\sum}w_{i}
\end{equation}

\begin{equation}
\underset{1\leq i\leq h}{\sum}w_{i}\geq\frac{1}{2}\underset{1\leq i\leq k}{\sum}w_{i}
\end{equation}
\end{center} 

and therefore choose the variables as follows:

\begin{center}
$x_{i}=1,\, i=1..h,\: x_{i}=0,\, i=h+1..k$
\end{center} 

With this choice of the values of the $k$ variables we may have (note that $E>0$ and $k\geq2$ if $e\geq3$):

\begin{center}
$\underset{1\leq i\leq k}{\sum}w_{i}\cdot x_{i}=\underset{1\leq i\leq h}{\sum}w_{i}\geq\frac{1}{2}\underset{1\leq i\leq k}{\sum}w_{i}>\frac{1}{2}\underset{1\leq i\leq k}{\sum}w_{i}-E$
\end{center} 

And therefore the second constraint (12) is satisfied.

Regarding the first constraint (11) it holds if:

\begin{center}
\begin{equation}
\frac {1}{2}\underset{1\leq i\leq k}{\sum} w_{i} \geq \underset{1\leq i\leq h}{\sum} w_{i} - E
\end{equation}
\end{center}

Note that this inequality holds for the number $e=16$ and the choice of the $w= [15, 30, 12] $ (see the sidebar n.5) but not for the number $e=68$ and a choice of the $w = [1155, 4620, 1386, 330, 210]$ (see the sidebar n.1).  Conversely it holds for the choice $w = [3465, 0, 1386, 990, 1050]$ (that is $d=39$).\\
On the other hand let observe that we may order the terms $w_{i}$ such that: $w_{i}\leq w_{i+1},\forall i=1..k-1$, and then we can determine the index $h$ such that:
\begin{center}

\begin{equation}
\underset{1\leq i\leq h}{\sum}w_{i}<\frac{1}{2}\underset{1\leq i\leq k}{\sum}w_{i}
\end{equation}

\begin{equation}
\underset{1\leq i\leq h+1}{\sum}w_{i}\geq\frac{1}{2}\underset{1\leq i\leq k}{\sum}w_{i}
\end{equation}
\end{center} 

and therefore choose the variables as follows:
\begin{center}
$x_{i}=1,\, i=1..h,\: x_{i}=0,\, i=h+1..k$
\end{center} 
With this choice of the values of the $k$ variables we may have (note that $E>0$ and $k\geq2$ if $e\geq3$):
\begin{center}
$\frac{1}{2}\underset{1\leq i\leq k}{\sum}w_{i}>\underset{1\leq i\leq h}{\sum}w_{i}=\underset{1\leq i\leq k}{\sum}w_{i}\cdot x_{i}>\underset{1\leq i\leq k}{\sum}w_{i}\cdot x_{i}-E$
\end{center} 
And therefore the first constraint (11) is satisfied.
Regarding the second constraint (12) it holds if:

\begin{center}
\begin{equation}
\frac {1}{2}\underset{1\leq i\leq k}{\sum} w_{i} \leq \underset{1\leq i\leq h}{\sum} w_{i} + E
\end{equation}
\end{center}

In other terms, both these conditions now obtained  (15 and 18) could define a particular heuristic in finding the appropriate set of $w_{i}$ together with the particular choices of the sign (that is $x_{i}$, as in 13 and 14, or in 16 and 17). We will examine such an algolirithm in a next research.\\


Furthemore, instead of deriving a generic value for the quantity d, it could be required to identify the greatest or the smallest d with the properties we have dicussed in the preceding. The natural way to do this is to write the related optimization problem in terms of integer linear program, and we can see that it is a special case of well known formulation classified in literature as 'Knapsack Problem' \cite{key-106}. In partcular, once chose the numbers $b_{i}$:
\begin{center}
maximize (or minimize) $W_{(1\times k)}X_{(k\times1)}$

subject to:

$WX\leq U$ 

$-WX\leq U'$ 

$x_{i}\in\left\{ 0,1\right\}$
\end{center}

where:
\begin{center}
$U=\frac{1}{2}\underset{1\leq i\leq k}{\sum}w_{i}+E,$ \\
$U'=-\frac{1}{2}\underset{1\leq i\leq k}{\sum}w_{i}+E,$ \\
$W_{(1\times k)}=(w_{1}\:...\: w_{k}),$ \\
$X_{(k\times1)}=(x_{1}\:...\: x_{k})^{T}$.
\end{center}

Given a set of items, each with a weight and a value (in our case they are coincident and the problem is called 'Subset sum problem'), determine the number of each item to include in a collection so that the total weight is less than or equal to a given limit and the total value is as large as possible. It derives its name from the problem faced by someone who is constrained by a fixed-size knapsack and must fill it with the most valuable items. Moreover, considering the two sets of constrains, we observe that various methods are known in literature in order to deal with negative weight \cite{key-107}. Let see as an example the following sidebar where the calculus is provided for the integer 16.


\section{Conclusion and future work}
In the present article we have reconducted the problem of finding a pair of mirror primes respect to a given integer $e>7$, at a well known problems of operations research. We have provided some algorithms to solve it as well. We derived a condition in terms of numbers $w_{i}$ and $e$ that can lead to a more efficient way to choose the appropriate set $w_{i}$. An interesting way of further research is both theoretical and pratical. It will be of high interest the completion of the proof of statement 4 in order to have an upper bound in the number of elements in the search domain. Moreover it will be useful to have a detailed analysis of a more efficient algorithm based on a new heuristic, as stated in the last section, exploiting the conditions (15) and (18).


\newpage
\framebox{\begin{minipage}[t]{1\columnwidth}%

\[
\\e=16
\]

\begin{center}
\begin{tabular}{|c|c|c|c|}
\hline 
$p_{i}$ & $I=I{}_{S} \cup I{}_{R}$ & $\{b_{j}\}_{i=1..k}$ \tabularnewline
\hline 
\hline 
2 & $I{}_{S}$ & $\left\{ \pm 1, \pm 3, \pm 5, ...\right\} $ \tabularnewline
\hline 
3 & $I{}_{R}$ & $\left\{ 0, \pm 3, \pm 6, ...\right\} $ \tabularnewline
\hline 
5 & $I{}_{R}$ & $\left\{0, \pm 2, \pm 3, \pm 5, ...\right\} $ \tabularnewline
\hline 
\end{tabular}
\par\end{center}

\begin{center}
from the CRT:
\par\end{center}

\begin{center}
\begin{tabular}{|c|c|c|}
\hline 
$p{}_{i}$ & $P{}_{i}$ & $P{}_{i}'$\tabularnewline
\hline 
\hline 
2 & 15 & 1 \tabularnewline
\hline 
3 & 10 & 1 \tabularnewline
\hline 
5 & 6 & 1 \tabularnewline
\hline 
\end{tabular}
\par\end{center}

\[
d{}_{h}=\pm b_{1}\cdot15\pm b_{2}\cdot10\pm b_{3}\cdot6; 1\leq h\leq \underset{1\leq j\leq k}{\prod}|\{b_{j}\}_{i}|
\]

\begin{center}
for example consider the following choices:
\par\end{center}

\begin{center}
\[
B=\begin{bmatrix}1 & 0 & 0\end{bmatrix}, d=15 > e-2
\]
 solution not feasible 
\par\end{center}

\begin{center}
\[
B=\begin{bmatrix}-1 & 3 & -2\end{bmatrix}, d=3 \le e-2
\]
 solution feasible
\par\end{center}

\begin{center}
therefore if we can choose $d=3$:
\par\end{center}

\[
16-3=13\:,\: prime
\]
\[
16+3=19\:,\: prime
\]

\begin{center}
so that:
\par\end{center}

\[
13+19=32=2\cdot16
\]
\end{minipage}}
\begin{center}{\label{deepbf}Sidebar 5. How obtain the number $d$ in the case of $e=16$}
 \end{center}


\end{document}